\documentclass[12pt]{article}
\usepackage[margin=2.25cm]{geometry}

\usepackage{amsmath}
\usepackage{amssymb}
\usepackage{amsthm}
\usepackage{amsfonts}
\usepackage{enumerate}
\usepackage{graphicx}
\usepackage{color}
\usepackage{verbatim}
\usepackage[colorlinks=true,citecolor=black,linkcolor=black,urlcolor=blue]{hyperref}

\usepackage{multicol}
\usepackage{mathrsfs}
\usepackage{mathtools}

\usepackage{ytableau}

\usepackage{subfig}
\usepackage{tikz}

\usetikzlibrary{arrows,shapes,positioning,decorations.markings}

\tikzset{->-/.style={decoration={
  markings,
  mark=at position .5 with {\arrow{>}}},postaction={decorate}}}

\usepackage{array}

\usepackage{xcolor,colortbl}

\usepackage{soul}

\newtheorem{theorem}{Theorem}[section]

\newtheorem{lemma}[theorem]{Lemma}
\newtheorem{example}[theorem]{Example}

\newtheorem{conjecture}[theorem]{Conjecture}

\newtheorem{proposition}[theorem]{Proposition}

\def\Q{\mathbb{Q}}

\newcommand{\inv}{n_{inv}}

\title{Combinatorics of Tableau Inversions}

\author{Jonathan E. Beagley \qquad Paul Drube\\
\small Department of Mathematics and Statistics\\
\small Valparaiso University\\
\small \tt \{jon.beagley, paul.drube\}@valpo.edu 
}

\begin{document}
\maketitle

\begin{abstract}
A tableau inversion is a pair of entries in row-standard tableau $T$ that lie in the same column of $T$ yet lack the appropriate relative ordering to make $T$ column-standard.  An $i$-inverted Young tableau is a row-standard tableau along with a precisely $i$ inversion pairs.  Tableau inversions were originally introduced by Fresse to calculate the Betti numbers of Springer fibers in Type A, with the number of $i$-inverted tableaux that standardize to a fixed standard Young tableau corresponding to a specific Betti number of the associated fiber.  In this paper we approach the topic of tableau inversions from a completely combinatorial perspective.  We develop formulas enumerating the number of $i$-inverted Young tableaux for a variety of tableaux shapes, not restricting ourselves to inverted tableau that standardize a specific standard Young tableau, and construct bijections between $i$-inverted Young tableaux of a certain shape with $j$-inverted Young tableaux of different shapes.  Finally, we share some the results of a computer program developed to calculate tableaux inversions.

\end{abstract}

\section{Introduction}
\label{sec: intro}

Let $\lambda = (\lambda_1,\lambda_2,\hdots,\lambda_m)$ be a non-increasing sequence of positive integers that partition $N$.  A \textbf{Young diagram} $Y$ of shape $\lambda$ is a left-justified array of $N = \lambda_1 + \hdots + \lambda_m$ boxes such that there are $\lambda_i$ boxes in the i\textsuperscript{th} row of $Y$.  A standard filling of $Y$ is a bijective assignment of the integers $1,2,\hdots,N$ to the boxes of $Y$, producing what is known as a tableau of shape $\lambda$.  The resulting tableau is said to be \textbf{row-standard} if its entries are increasing from left-to-right along each row, and \textbf{column-standard} if its entries are increasing from top-to-bottom in each column.  A \textbf{standard Young tableau} of shape $\lambda$ is a tableau of shape $\lambda$ that is both row-standard and column-standard.  For more basic information on Young Tableaux, see \cite{Fulton}.

One established result about standard Young tableaux that we will make repeated use of is the Hook-Length formula.  Let $Y$ be a Young diagram of shape $\lambda$.  To each box in $Y$ one may assign a \textbf{hook-length} that equals one plus the number of boxes directly below or directly to the right of the given box (i.e.- the number of boxes in the ``hook"-shaped sub-tableaux whose corner lies at the chosen box).  As stated below, the Hook-Length formula uses these hook-lengths to determine the number of standard Young tableaux of shape $\lambda$.  For more discussion of this result, see Chapter 7 of \cite{Stanley2}.

\begin{theorem}[The Hook-Length Formula]
\label{thm: hook length formula}
Let $\lambda = (\lambda_1,\lambda_2,\hdots,\lambda_m)$ be a non-increasing partition of $N$, and let $h_{ij}$ denote the hook-length of the $(i,j)$-entry in a Young diagram of shape $\lambda$.  Then the number of standard Young tableaux of shape $\lambda$ is:

\begin{center}
$\displaystyle{\vert S(\lambda) \vert = \frac{N!}{\prod h_{ij}}}$
\end{center}
\end{theorem}

Now take a permutation $\sigma \in S_n$.  An inversion of $\sigma$ is a pair of integers $1 \leq i,j \leq n$ such that $i < j$ yet $\sigma(i) > \sigma(j)$.  We write $(i,j)_\sigma$ or simply $(i,j)$ if $i$ and $j$ form an inversion pair of $\sigma$ with $i < j$.  We denote the total number of inversions of $\sigma$ by $\inv(\sigma)$.

In \cite{Fresse}, Fresse generalized this notion of inversion to multi-column tableaux.  Let $\tau$ be a row-standard tableau of shape $\lambda$.  A pair of entries $1 \leq i,j \leq N$ from the same column of $\tau$ are an \textbf{inversion} of $\tau$ if $i < j$ and one of the following holds:

\begin{enumerate}
\item Either $i$ or $j$ lacks a entry to its right and $i$ is below $j$.
\item $i$ is bordered immediately on its right by $i'$, $j$ is bordered immediately on its right by $j'$, and $i' > j'$.
\end{enumerate}

As above, we write $(i,j)_\tau$ or simply $(i,j)$ if $i$ and $j$ constitute an inversion pair of $\tau$ with $i<j$, while we denote the total number of inversion pairs of $\tau$ by $\inv(\tau)$.  Notice that a row-standard tableau $\tau$ is a standard Young tableau if and only if $\inv(\tau) = 0$, and that this definition specializes to the notion of permutation inversion if you represent $\sigma \in S_n$ by the single-column tableau whose entries appear in the order $\sigma(1),\hdots,\sigma(n)$.

Given any row-standard tableau $\tau$ of shape $\lambda$, one may independently reorder each of the columns of $\tau$ to produce a tableau that is column-standard.  As argued in \cite{Fresse}, the resulting tableau is also guaranteed to be row-standard, and hence is a standard Young tableau of shape $\lambda$.  This standard Young tableau is clearly unique, and is referred to as the \textbf{standardization} of $\tau$, denoted $st(\tau)$; we equivalently say that $\tau$ is ``based" on the standard Young tableau $st(\tau)$\footnote{We note that this is a distinct definition from standardization of a semi-standard Young tableau, as found in the literature.}.  As any row-standard tableau may be ``standarized" by removing its inversions, we henceforth refer to the collection of all row-standard tableaux of shape $\lambda$ as \textbf{inverted (standard) Young tableaux} of shape $\lambda$.  We denote the set of all inverted Young tableaux of shape $\lambda$ by $I(\lambda)$.

Figure \ref{fig: basic inversion example} shows an inverted Young tableau of shape $\lambda = (3,3,3)$ alongside its standardization.  Throughout this paper, we will highlight the location of inversions with shaded boxes, although do notice that this convention becomes slightly ambiguous when there are multiple inversions per column.

\begin{figure}[h!]
\centering
\begin{ytableau}
1 & *(gray!50) 2 & *(gray!50) 8 \\
*(gray!50) 4 & *(gray!50) 5 & *(gray!50) 6 \\
*(gray!50) 3 & 7 & 9
\end{ytableau}
\hspace{.03in}
\raisebox{-12pt}{$\rightarrow$}
\hspace{.03in}
\begin{ytableau}
1 & 2 & 6 \\
3 & 5 & 8 \\
4 & 7 & 9
\end{ytableau}
\caption{An inverted tableau $\tau$ with inversions $(3,4),(2,5),(6,8)$ and its standardization $st(\tau)$}
\label{fig: basic inversion example}
\end{figure}

In \cite{Fresse}, Fresse utilizes inverted Young tableaux to determine the Betti numbers of Springer fibers in type A.  In particular, he fixes a specific standard Young tableau $T$ of shape $\lambda$ and considers the corresponding set of inverted tableaux $inv(T) = \lbrace \tau \ \vert \ st(\tau) = T \rbrace$.  He then argues that the component of the Springer variety $F_\lambda$ corresponding to $T$ has $m$\textsuperscript{th} Betti number equaling the number of inverted Young tableaux in $inv(T)$ with precisely $d-m$ inversions, where $d$ is the dimension of the entire Springer variety.  Ranging over all standard Young tableaux $T$ then gives:

\begin{theorem}[Fresse]
\label{thm: Fresse}
Let $\lambda$ be a partition of $N$ and consider the Springer variety $F_\lambda$.  If $dim(F_\lambda)=d$, then the m\textsuperscript{th} Betti number $b_m = dim(H^{2m}(F_\lambda,\Q))$ equals the number of inverted Young tableaux $\tau$ of shape $\lambda$ with $\inv(\tau) = d-m$.
\end{theorem}

For a given shape $\lambda = (\lambda_1,\hdots,\lambda_m)$, let $S_i(\lambda) = S_i(\lambda_1,\hdots,\lambda_m)$ denote the set of row-standard tableaux $\tau$ of shape $\lambda$ with $\inv(\tau) = i$.  Theorem \ref{thm: Fresse} then states that $b_m = \vert S_{d-m}(\lambda) \vert$.  We henceforth refer to elements of $S_i(\lambda)$ as \textbf{i-inverted (standard) Young tableau} of shape $\lambda$.  Note that standard Young tableaux correspond to $S_0(\lambda) = S(\lambda)$, and that the size of the set $S_0(\lambda)$ is explicitly determined by the Hook-Length formula.

This paper won't assume detailed knowledge of Springer varieties.  We present Fresse's results primarily to motivate our study of tableaux inversions.  Our approach is purely combinatorial, and is concerned with developing tractable methods for enumerating elements in the sets $S_i(\lambda)$ for arbitrary $\lambda$ and any $i \geq 0$.

The primary difficulty in adapting Fresse's approach to these goals is that his enumerative results require knowledge of specific standard Young tableau as a starting point.  Even if one can easily produce a generating function that gives the sizes of the sets $S_i(\lambda) \cap inv(T)$ for a fixed standard tableau $T$, there is no workable way for determining all possible generating functions when one ranges over all underlying $T$.  We avoid this difficulty by presenting techniques for directly calculating the sizes of the entire sets $S_i(\lambda)$.







\subsection{Outline of Results}
\label{subsec: results outline}

In this paper we begin with a series of basic results about the number of inverted Young Tableaux.  We give a closed formula (Proposition \ref{thm: total number of inverted tableaux}) for the total number of inverted tableaux for a fixed shape $\lambda$ in terms of its row lengths.  Then we show that, for any shape $\lambda$, there is a maximal inversion number $M_\lambda$ that is determined by sums of triangular numbers, and that there exists a single tableau of shape $\lambda$ realizing that maximal number of inversions (Proposition \ref{thm: maximal inversion number}).  After directly enumerating all one-column and two-row tableaux with a fixed number of inversions (Theorem \ref{thm: two-row tableaux inversions}) , we proceed to our primary combinatorial results.  Our first major result, introduced as Theorem \ref{thm: enumeration, rectangular i=1} (and generalized to non-rectangular tableaux in Theorem \ref{thm: enumeration, general i=1}), relates the number of rectangular tableaux with precisely one inversion to the number of standard Young tableaux of a related ``stair-step" shape:

\begin{theorem}
\label{thm: preview 1}
Let $n,m \geq 1$, and consider the m-row shapes $\lambda = (n,\hdots,n)$, $\widetilde{\lambda} = (n+1,n,\hdots,n,n-1)$.  Then $\vert S_1(\lambda) \vert = \vert S_0(\widetilde{\lambda}) \vert$. 
\end{theorem}

We then turn to tableaux with more than one inversion.  Working downward from the maximum inversion number $M_\lambda$, Theorem \ref{thm: M-1 rectangular} and \ref{thm: M-2 rectangular} combine to prove the following closed formulas enumerating the number of $m \times n$ rectangular tableaux with precisely $M_\lambda - 1$ and $M_\lambda - 2$ inversions:

\begin{theorem}
\label{thm: preview 2}
Take $n \geq 1, m \geq 3$ and consider the m-row rectangular shape $\lambda=(n,\hdots,n)$.  Then $\vert S_{M_\lambda - 1} (\lambda) \vert = mn - 1$ and $\vert S_{M_\lambda - 2} (\lambda) \vert = \frac{(mn-2)(mn+1)}{2}$
\end{theorem}

Closing the paper are a series of more specific results and conjectures enumerating inverted tableaux whose inversion number is ``sufficiently large" (a realm that we refer to as the ``tail end" of a shape's inversion distribution). All results and conjectures in this paper were developed with the help of a computer program developed by the authors.  Details about this program, as well as several tables of generated results, may be found in Appendix \ref{sec: inversion tables}.

\section{Basic Enumerative Results About Tableau Inversions}
\label{sec: basic results}

In this section, we present a number of foundational enumerative results involving tableau inversions that are independent from the more involved methods of Sections \ref{sec: enumeration i=1} and \ref{sec: enumeration i>1}.

\subsection{Total Number of Inverted Young Tableaux of Shape $\lambda$}
\label{subsec: total number of inverted tableaux}

Let $\lambda = (\lambda_1,\lambda_2,\hdots,\lambda_m)$ be a non-increasing partition of the positive integer $N$, and consider the set $I(\lambda) = \bigcup_i S_i(\lambda)$ of inverted Young tableaux of shape $\lambda$ (with any number of inversions $i \geq 0$).

$I(\lambda)$ is the collection of all tableaux of shape $\lambda$ that are row-standard.  As there is no restriction on the columns of these tableaux, and since there is a unique way to order each row of a tableau so that it is increasing, elements of $I(\lambda)$ are in bijection with ordered partitions of $\lbrace 1,\hdots,N \rbrace$ into $m$ subsets of respective sizes $\lambda_1,\hdots,\lambda_m$.  This immediately proves the following.

\begin{proposition}
\label{thm: total number of inverted tableaux}
Let $\lambda = (\lambda_1,\lambda_2,\hdots,\lambda_m)$ be a non-increasing sequence of positive integers.  Then the total number of inverted Young tableaux of shape $\lambda$ is:

\begin{center}
$\displaystyle{\vert I(\lambda) \vert = \sum_{i=0}^\infty \vert S_i(\lambda) \vert = \binom{\lambda_1 + \lambda_2 + \hdots + \lambda_m}{\lambda_m} \binom{\lambda_1 + \lambda_2 + \hdots + \lambda_{m-1}}{\lambda_{m-1}} \hdots \binom{\lambda_1 + \lambda_2}{\lambda_2} \binom{\lambda_1}{\lambda_1} }$
\end{center}
\end{proposition}

In terms of Fresse's results in \cite{Fresse}, Proposition \ref{thm: total number of inverted tableaux} states that the sum of all Betti numbers of $F_\lambda$ is given by $\vert I(\lambda) \vert$.  As one application of this straightforward result, for $\lambda= (n,n)$ the quantity $\vert I(\lambda) \vert$ gives the rank of the $sl_2$ skein module of surfaces over the solid torus with $2n$ boundary points \cite{Russell}.  In an upcoming paper \cite{BDR}, one of the authors proves that $\vert I(\lambda) \vert$ with $\lambda = (n,n,n)$ gives the rank of the $sl_3$ skein module of surfaces over the solid torus with $2n$ boundary points.

\subsection{Maximum Number of Inversions for Shape $\lambda$}
\label{subsec: maximal number of inversions}

A natural question to arise from consideration of Proposition \ref{thm: total number of inverted tableaux} is how many nonzero $\vert S_i(\lambda) \vert$ appear in the summation for $\vert I(\lambda) \vert$.  Clearly, a row-standard tableau with a finite number of entries can only admit a finite number of inversions, as each entry can only be involved in a maximum of one inversion pair with every other entry in its column.  To determine ``maximum inversion number", we consider column heights in inverted Young tableaux of shape $\lambda$.

First consider the case of single-column tableau with $N$ total entries, so that $\lambda = (1,1,\hdots,1)$.  The maximum number of inversions for a tableau of this shape occurs when the entries appear in descending order $N,\hdots,2,1$.  In this situation we have $N-1$ inversion pairs $(1,2),(1,3),\hdots,(1,N)$ where $1$ is the smaller entry, $N-2$ inversion pairs $(2,3),\hdots,(2,N)$ where $2$ is the smaller entry, etc.  It follows that the maximum number of inversions for a tableau of shape $\lambda$ is $T_{N-1} = 1+2+\hdots+(N-1) = \binom{N}{2}$, which is the $(N-1)$\textsuperscript{st} triangular number.  Notice that the ``reverse order" tableau described above is the unique inverted Young tableau of this shape realizing the maximum possible inversion number $T_{N-1}$.

Generalizing to tableaux of arbitrary shape $\lambda$ is essentially a repeated application of the procedure above:

\begin{proposition}
\label{thm: maximal inversion number}
Let $\lambda = (\lambda_1,\hdots,\lambda_m)$, and define $h_j = \vert \lbrace \lambda_i \ \vert \ \lambda_i \geq j \rbrace \vert$ to be the height of the j\textsuperscript{th} column for any tableau of shape $\lambda$.  Then the maximum number of inversions for any inverted Young tableau of shape $\lambda$ is:

\begin{center}
$\displaystyle{M_\lambda = \sum_j T_{h_j-1} = \sum_j \binom{h_j}{2}}$
\end{center}

\noindent Moreover, this maximum inversion number is realized by precisely one inverted Young tableau of shape $\lambda$, so that $\vert S_{M_\lambda}(\lambda) \vert = 1$.

\end{proposition}
\begin{proof}
We have already shown that the maximum number of inversions within a single column of height $h_j$ is $T_{h_j-1}$, so when we range over all columns of our tableau we clearly can't obtain more than $\sum_j T_{h_j-1}$ total inversions.  Thus we merely need to construct a tableau that exhibits this maximum number of inversions.

Let $N = \lambda_1 + \hdots + \lambda_m$.  We work from right-to-left through the columns of the tableau.  For the rightmost ($n$\textsuperscript{th}) column, place $N,N-1,\hdots N-h_n$ in decreasing order from top-to-bottom as in the single column case.  There are then $T_{h_n-1}$ inversions involving elements of this column.  Moving one column to the left, we place the next $h_{n-1}$ largest remaining entries in the unique order that guarantees $T_{h_{(n-1)}-1}$ inversions in its column.  This is accomplishing by ``flipping" the order of the column relative to the column on its right, with the important convention that any empty boxes on the right are filled with arbitrary large numbers that increase from top-to-bottom, so that the smallest entry in the leftward column is directly to the right of the largest entry in the rightward column, etc.  Continuing to work leftward through the columns, we repeat this procedure by placing the $h_j$ largest remaining entries in the unique order that guarantees $T_{h_j - 1}$ inversions in the $j$\textsuperscript{th} column.  An example of this procedure for a non-rectangular shape is shown in Figure \ref{fig: maximal inversion number example}.

Notice that this resulting tableau is guaranteed to be row-standard, as every entry in a given column is smaller than every entry in columns to its right.  Also notice that this is the only possible tableau with the maximum of $\sum_j T_{h_j-1}$ inversions: changing the ordering within any column reduces the number of inversions in that column, and switching entries between columns is guaranteed to reduce the number of inversions possible in the leftward column (as there is now as entry in that leftward column that is larger than some entry in the rightward column).
\end{proof}

\begin{figure}[h!]
\centering
\begin{ytableau}
*(gray!50) 2 & *(gray!50) 7 & *(gray!50) 10 \\
*(gray!50) 1 & *(gray!50) 8 & *(gray!50) 9 \\
*(gray!50) 3 & *(gray!50) 6 \\
*(gray!50) 4 & *(gray!50) 5
\end{ytableau}
\caption{The unique inverted Young tableau of shape $\lambda = (3,3,2,2)$ with the maximum number of $T_{3}+T_{3}+T_{1} = 13$ inversion pairs.  When giving the middle column the ``reverse" ordering relative to the rightmost column, we treat the empty slots in the lower-right as 11 and 12.}
\label{fig: maximal inversion number example}
\end{figure}

For a quick specialization to rectangular tableaux of size $m \times n$, note that Proposition \ref{thm: maximal inversion number} gives a maximum inversion number of $M_{n,\hdots,n} = n \binom{m}{2}$.  In particular, two-row rectangular tableaux have maximal inversion number $M_{n,n}= n$ equal to their number of columns.

\subsection{Enumerating i-Inverted Young Tableaux of Shape $\lambda = (1,1,\hdots,1)$}
\label{subsec: enumerating one-column inverted tableaux}

In general, it is extremely difficult to determine the number of inverted tableaux of shape $\lambda$ with precisely $i$ inversions, but there are several basic shapes $\lambda$ for which this is computable without the complicated methods of Sections \ref{sec: enumeration i=1} and \ref{sec: enumeration i>1}.  The first of these shapes are single-column tableaux, so that $\lambda = (1,1,\hdots,1)$.  In this case, tableau inversions are identical to ordinary permutation inversions, which are well-studied in the literature.

The number of length-$m$ permutations with precisely $i$ inversions, and hence the number of tableaux of size $m \times 1$ with precisely $i$ inversions,  is given by the Mahonian number $M(m-1,i)$ \cite{Stanley1}.  For a fixed $m \geq 1$, in \cite{Stanley1} it is also shown the Mahonian numbers have generating function:

\begin{equation}
\label{eq: Mahonian numbers}
\sum_{i=0}^\infty M(m-1,i)x^i = \prod_{j=0}^{m-1} \sum_{k=0}^j x^k = (x+1)(x^2 + x + 1) \hdots (x^{m-1} + \hdots + x + 1)
\end{equation}

Notice that the generating polynomial of Equation \ref{eq: Mahonian numbers} is of degree $T_{m-1}$ and that it is always monic, corroborating our results from Proposition \ref{thm: maximal inversion number}.

\subsection{Enumerating i-Inverted Young Tableaux of Shape $\lambda = (n,n)$}
\label{subsec: enumerating two-row inverted tableaux}

Although we can no longer rely upon previously-established results about permutation inversions, the other situation where it is still tractable to directly compute the number of i-inverted tableaux is with two-row rectangular tableaux.  We immediately jump to the two-row case because one-row tableaux are entirely trivial: if $\lambda = (n)$ then $\vert S_0(\lambda) \vert = 1$ and $\vert S_i(\lambda) \vert = 0$ for all $i \geq 1$, as inversions require at least two elements in a column.

Before proceeding to the two-row case, we establish some new terminology.  An $m \times n$ rectangular tableau $\tau$ is said to be a \textbf{(vertically) split tableau} if the first $j$ columns of $\tau$ (where $j < n$) contain the first $jm$ entries of $\tau$.  Hence, a split tableau is merely two disjoint tableaux sitting side-by-side, with the entries of the rightward tableau re-indexed to continue where the leftward tableau stopped.  Notice that the number of split (standard Young) tableaux of size $m \times n$ that split after the $j$\textsuperscript{th} column equals the number of (standard Young) tableaux of size $m \times j$ times the number of (standard Young) tableaux of size $m \times (n-j)$.  Also notice that an inverted Young tableau $\tau$ splits after its $j$\textsuperscript{th} column iff its standardization $st(\tau)$ splits after its $j$\textsuperscript{th} column.

It is well-known that the number of rectangular standard Young tableaux of size $2 \times n$ equals the $n$\textsuperscript{th} Catalan number $C_n$, a fact that can be directly calculated from the Hook-Length Formula of Theorem \ref{thm: hook length formula} or by setting up a bijection with noncrossing matchings.  This means that $\vert S_0(n,n) \vert = C_n$.

To find $\vert S_i(n,n) \vert$ for $i > 0$, we begin by noting that each column in a two-row inverted Young tableau has either $0$ or $1$ inversion pair.  If a column admits an inversion pair, this means that the larger entry in that column containing the inversion much be smaller than the smallest of the two entries in the column directly to the right.  It follows that a two-row inverted Young tableau must split after any column in which it has an inversion.  For an example of this phenomenon see Figure \ref{fig: inversions split tableaux}, where larger subscripts correspond to larger entries.  In that example, $b_2 < c_1$ ensures that the tableau must split.

\begin{figure}[h!]
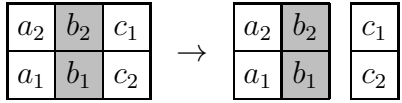

\centering
\begin{ytableau}
a_2 & *(gray!50) b_2 & c_1 \\
a_1 & *(gray!50) b_1 & c_2 \\
\end{ytableau}
\hspace{.03in}
\raisebox{-4pt}{$\rightarrow$}
\hspace{.03in}
\begin{ytableau}
a_2 & *(gray!50) b_2 \\
a_1 & *(gray!50) b_1 \\
\end{ytableau}
\hspace{.01in}
\begin{ytableau}
c_1 \\
c_2 \\
\end{ytableau}
\caption{An Inversion Guarantees a Split Two-Row Tableau}
\label{fig: inversions split tableaux}
\end{figure}

\begin{theorem}
\label{thm: two-row tableaux inversions}
Let $\lambda = (n,n)$ for some $n \geq 1$.  Then the number of inverted Young tableaux of shape $\lambda$ with precisely $i$ inversions is:

\begin{center}
$\displaystyle{\vert S_i(n,n) \vert = \left( \sum_{k_1 + \hdots + k_i = n} C_{k_1} C_{k_2} \hdots C_{k_i} \right) \ + \ \left( \sum_{l_1 + \hdots + l_{i+1} = n} C_{l_1} C_{l_2} \hdots C_{l_{i+1}} \right) } $
\end{center}

\noindent Where $C_j$ is the k\textsuperscript{th} Catalan number, and the summations run over all ordered partitions of $n$ of length $i$ and of length $i+1$, respectively. 
\end{theorem}
\begin{proof}
First observe that, if $i = 0$, then the first summation is empty and the second summation contains the single summand $C_n$ corresponding to the unique length-one partition $(n)$ of $n$.  Thus $\vert S_0(n,n) \vert = C_n$, as detailed above.

For $i \geq 1$, consider the subset of $i$-inverted two-row tableaux $\tau$ where the $i$ inversions are located at columns $j_1,\hdots,j_i$.  As mentioned above, any such tableau must vertically split after each of these columns.  If $j_i = n$, the tableau splits into $i$ sub-tableaux $\tau_0,\hdots, \tau_{i-1}$ such that $\tau_q$ ends after column $j_{q+1}$.  If $j_i < n$, the tableau splits into $i+1$ sub-tableaux $\tau_0,\hdots,\tau_i$ such that $\tau_q$ similarly ends after column $j_{q+1}$ for $q<i$ and $\tau_i$ ends after column $n$.  If $j_i = n$, we may conclude that the number of such tableaux is $C_{j_1}C_{j_2 - j_1}\hdots C_{j_i - j_{(i-1)}}$, which is a summand from the first summation in the Theorem.  If $j_i < n$, we may conclude that the number of such tableaux is $C_{j_1}C_{j_2 - j_1}\hdots C_{n - j_i}$, which is a summand from the second summation in the Theorem.  Noting that any ordered partition of $n$ of length $i$ or $i+1$ uniquely corresponds to a selection of $i$ columns from our tableau (including the $n$\textsuperscript{th} column or not including the n\textsuperscript{th} column, respectively), ranging over all possible partitions gives the full summation for $\vert S_i(n,n) \vert$ in the Theorem. 
\end{proof}

\begin{example}
\label{ex: two-row inversion numbers}

For $n=3$, Theorem \ref{thm: two-row tableaux inversions} gives:

\vspace{.05in}

\centering
\begin{tabular}{| >{$}c<{$} | >{$}l<{$} | >{$}l<{$} |}
\hline i & \lambda \vdash 3 & \vert S_i(3,3) \vert \\ \hline
0 & (3) & 5 \\
1 & (3),(2,1),(1,2) & 5 + 2*1 + 1*2 = 9 \\
2 & (2,1),(1,2),(1,1,1)& 2*1 + 1*2 + 1*1*1 = 5\\
3 & (1,1,1) & 1*1*1 = 1\\ \hline
\end{tabular}
\end{example}

Readers in search of additional combinatorial identities should note that, when successively ranging over all $n > 0$, the sequence formed by the $\vert S_i(n,n) \vert$ is A039599 on OEIS \cite{OEIS}.

In light of Theorem \ref{thm: Fresse}, Theorem \ref{thm: two-row tableaux inversions} applies to give a closed formula for all Betti numbers of the Springer variety $F_\lambda$ when $\lambda = (n,n)$.  Also note that, as $(1,\hdots,1)$ is the only partition of $n$ of length at least $n$, then $M_{n,n} = n$ and $\vert S_n(n,n) \vert = 1$ as predicted by Proposition \ref{thm: maximal inversion number}.

\section{Enumerating $1$-Inverted Young Tableaux}
\label{sec: enumeration i=1}

For the remainder of this paper, we introduce a number of techniques that allow us to calculate the number of inverted Young tableaux $\vert S_i(\lambda) \vert$ with precisely $i$ inversions for a variety of shapes $\lambda$.  One of our primary techniques will be to develop bijections between $i$-inverted tableaux of shape $\lambda$ and $j$-inverted tableaux of some other shape $\widetilde{\lambda}$.  In some cases, this second set of inverted Young tableaux will then be directly enumerable via the Hook-length formula or the methods of Section \ref{sec: basic results}.  In this particular section we begin by restricting our attention to tableaux with precisely one inversion.  We present an algorithm for directly calculating $\vert S_1(\lambda) \vert$ for rectangular tableaux of shape $\lambda = (n,\hdots,n)$.  We then offer a direct generalization to arbitrary (non-rectangular) shapes $\lambda$. 

\subsection{Calculating $\vert S_1(\lambda) \vert$, Rectangular Case}
\label{subsec: enumeration, rectangular i=1}

Let $\lambda = (n,\hdots,n)$, with $m$ consecutive $n$'s, so that we are dealing with rectangular tableaux of size $m \times n$.  The goal of this subsection is to develop a bijection that proves the following.

\begin{theorem}
\label{thm: enumeration, rectangular i=1}
Let $n,m \geq 1$, and consider the m-row shapes $\lambda = (n,\hdots,n)$, $\widetilde{\lambda} = (n+1,n,\hdots,n,n-1)$.  Then $\vert S_1(\lambda) \vert = \vert S_0(\widetilde{\lambda}) \vert$. 
\end{theorem}

Theorem \ref{thm: enumeration, rectangular i=1} relates the number of 1-inverted rectangular Young tableaux of size $m \times n$ to the number of standard Young tableaux of a certain ``stair-step" shape $\widetilde{\lambda}$.  $\widetilde{\lambda}$ is obtained from $\lambda$ by removing the box at the lower-right corner of our $m \times n$ rectangular diagram and moving it to the top of a new, $(n+1)$\textsuperscript{st} column.  See Figure \ref{fig: rectangular to stair-step comparison} for an example of this shape change.  Note that $\vert S_0(\widetilde{\lambda}) \vert$ is calculable using the Hook-Length Formula.

\begin{figure}[h!]
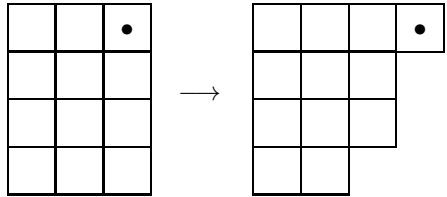

\centering
\begin{footnotesize}
\ydiagram{3,3,3,3}
*[\bullet]{0,0,0,2+1}
\hspace{.1in}
\raisebox{-18.5pt}{$\longrightarrow$}
\hspace{.08in}
\ydiagram{4,3,3,2}
*[\bullet]{3+1}
\end{footnotesize}
\caption{Shape Change in the Bijection of Theorem \ref{thm: enumeration, rectangular i=1} when $\lambda = (3,3,3,3)$}
\label{fig: rectangular to stair-step comparison}
\end{figure}

\begin{proof}[Proof of Theorem \ref{thm: enumeration, rectangular i=1}]
We outline two procedures that give well-defined maps \hfill \break $\phi_1: S_1(\lambda) \rightarrow S_0(\widetilde{\lambda})$, $\phi_2: S_0(\widetilde{\lambda}) \rightarrow S_1(\lambda)$, and then argue that those maps are inverses of one another.

For our map $\phi_1 : S_1(\lambda) \rightarrow S_0(\widetilde{\lambda})$, take a 1-inverted rectangular tableau $\tau \in S_1(\lambda)$.  Assume that the sole inversion pair of $\tau$ is $(a,b)$, where $a < b$, and that this inversion appears in the $k$\textsuperscript{th} column of $\tau$.  Notice that, since $\tau$ contains only one inversion, $b$ must appear directly above $a$ in the $k$\textsuperscript{th} column.  We begin our first procedure with the larger element $b$ of the inversion pair $(a,b)$.  We recursively ``bump" an increasing sequence distinguished elements $\lbrace b = c_k,c_{k+1},\hdots, c_n \rbrace$, one from each column of $\tau$ from the $k$\textsuperscript{th} column rightward, to the right by one column each.  Our procedure is as follows:

\begin{enumerate}
\item If $j<n$, $c_j$ isn't in the rightmost column of $\tau$.  Define $c_{j+1}$ to be the smallest entry $d$ in the $(j+1)$\textsuperscript{st} column of $\tau$ such that $d > c_j$.  Then ``bump" $c_j$ into the box occupied by $c_{j+1}$, temporarily producing a box in the $(j+1)$\textsuperscript{st} column that is shared by two elements.  There will then be an empty box in the $j$\textsuperscript{th} column: repeatedly fill any open boxes in the $j$\textsuperscript{th} column by sliding the smaller of the elements lying directly to the right or directly below that empty box into the box.  Do this until the open box has been moved into the $(j+1)$\textsuperscript{st} column.  Then repeat this entire procedure with $c_{j+1}$ until you reach $c_n$.
\item If $j=n$, $c_n$ lies is in the rightmost column of $\tau$.  Move $c_n$ to the top of a new $(n+1)$\textsuperscript{st} column of $\tau$.  There will be an empty box in the $n$\textsuperscript{th} column: slide all entries that are below that empty box up one slot, moving the empty box to the bottom of the column.
\end{enumerate}

An example of the procedure for $S_1(3,3,3) \hookrightarrow S_0(4,3,2)$ is shown in Figure \ref{fig: rectangular i=1 to i=0 example}.  This algorithm always transforms $\tau$ into a tableau $\widetilde{\tau}$ of shape $\widetilde{\lambda}$.  As there is a unique allowable move at each step, the resulting $\widetilde{\tau}$ is unique.  If $\tau$ begins with only a single inversion, the ``back-filling" component of the procedure from case \#1 ensures that $\widetilde{\tau}$ is a standard Young tableau.  Thus this algorithm describes a well-defined map $\phi: S_1(\lambda) \rightarrow S_0(\widetilde{\lambda})$.

\begin{figure}[h!]
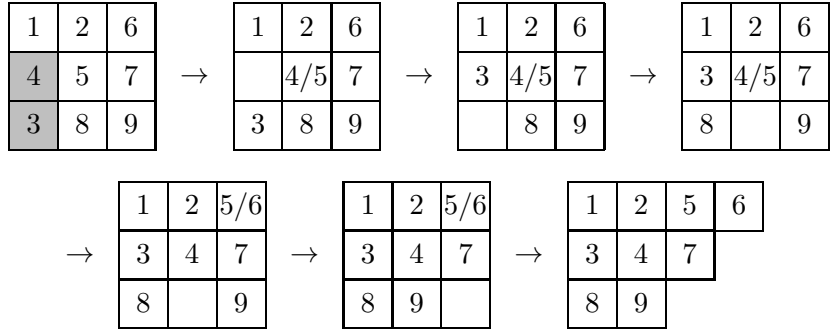

\centering
\begin{small}
\ytableausetup{boxsize=1.65em}
\begin{ytableau}
1 & 2 & 6 \\
*(gray!50) 4 & 5 & 7 \\
*(gray!50) 3 & 8 & 9
\end{ytableau}
\hspace{.03in}
\raisebox{-12pt}{$\rightarrow$}
\hspace{.03in}
\begin{ytableau}
1 & 2 & 6 \\
 & 4/5 & 7 \\
3 & 8 & 9
\end{ytableau}
\hspace{.03in}
\raisebox{-12pt}{$\rightarrow$}
\hspace{.03in}
\begin{ytableau}
1 & 2 & 6 \\
3 & 4/5 & 7 \\
 & 8 & 9
\end{ytableau}
\hspace{.03in}
\raisebox{-12pt}{$\rightarrow$}
\hspace{.03in}
\begin{ytableau}
1 & 2 & 6 \\
3 & 4/5 & 7 \\
8 &  & 9
\end{ytableau}

\vspace{.15in}

\raisebox{-12pt}{$\rightarrow$}
\hspace{.03in}
\begin{ytableau}
1 & 2 & 5/6 \\
3 & 4 & 7 \\
8 &  & 9
\end{ytableau}
\hspace{.03in}
\raisebox{-12pt}{$\rightarrow$}
\hspace{.03in}
\begin{ytableau}
1 & 2 & 5/6 \\
3 & 4 & 7 \\
8 & 9 & 
\end{ytableau}
\hspace{.03in}
\raisebox{-12pt}{$\rightarrow$}
\hspace{.03in}
\begin{ytableau}
1 & 2 & 5 & 6 \\
3 & 4 & 7 \\
8 & 9 
\end{ytableau}
\end{small}
\caption{$S_1(3,3,3) \hookrightarrow S_0(4,3,2)$ Example}
\label{fig: rectangular i=1 to i=0 example}
\end{figure}

For our second map $\phi_2: S_0(\widetilde{\lambda}) \rightarrow S_1(\lambda)$, take a standard tableau $T \in S_0(\widetilde{\lambda})$ and let $c_n$ be the sole entry in the $(n+1)$\textsuperscript{st} column of $T$.  Here our goal is to define a decreasing sequence of distinguished entries $\lbrace c_n,c_{n-1},\hdots \rbrace$, one from each column of $T$ beginning with the $(n+1)$\textsuperscript{st} column, and ``reverse bump" those elements one column to the left until one of the distinguished elements becomes the larger member of an inversion pair.  Here our procedure is as follows:

\begin{enumerate}
\item Consider $c_j$, which originally lies in the $(j+1)$\textsuperscript{st} column of $T$.  There will always be an empty box in the $j$\textsuperscript{th} column of $T$.  Repeatedly fill that empty box with the largest of $c_j$, the entry directly above the box, and the entry directly to the left of the box, stopping when the empty box is moved leftward into the $(j-1)$\textsuperscript{st} column of $T$ or when $c_j$ directly fills the empty box.
\item If the empty box is moved into the $(j-1)$\textsuperscript{st} column, define $c_{j-1}$ to be the largest entry $d$ in the $j$\textsuperscript{th} column such that $d < c_j$.  Then move $c_j$ into the box occupied by $c_{j-1}$, so that the tableau temporarily has a box containing two entries.  Now repeat step \#1 with $c_{j-1}$.
\item If $c_j$ directly fills an empty box in the $j$\textsuperscript{th} column and $a$ is the entry that lies directly above $c_j$ after this insertion, create a single inversion pair $(a,c_j)$ by flipping the rows containing $a$ and $c_j$ from the $j$\textsuperscript{th} column leftward.  The most important thing to note here is that the final tableau always admits the inversion $(a,c_j)$ at this step: i.e.- $c_j$ is guaranteed to be smaller than the entry directly to the right of $a$.  This is because, when $c_j$ fills an empty box, it always must move down by at least one row.  See Figure \ref{fig: rectangular i=0 proof explanation} for additional explanation of this fact.
\end{enumerate}

\begin{figure}[h!]
\centering
\begin{ytableau}
a_1 & b_1 & d_1 \\
a_2 & b_2 & \frac{c_j}{c_{j+1}} \\
a_3 &  & d_3 \\
a_4 & b_4 & d_4
\end{ytableau}
\caption{The $S_0(n+1,n,\hdots,n,n-1) \hookrightarrow S_1(n,\hdots,n)$ procedure at the $j$\textsuperscript{th} column.  As $b_2 < c_j$, at this stage either $c_j$ fills the empty box and an inversion is possible or $a_3$ is slid right and we return to step \#1 for the next column to the left.  It is not possible for $b_2$ to be slid down, which is the only way for $c_j$ not to slide down at least one row when filling an empty box.}
\label{fig: rectangular i=0 proof explanation}
\end{figure}

An example of this procedure for $S_0(4,3,2) \hookrightarrow S_1(3,3,3)$ is shown in Figure \ref{fig: rectangular i=0 to i=1 inverse example}.  Notice that the ``front sliding" move of step \#1 ensures that the tableau remains both row and column standard, so that the final tableau only has the single inversion that is introduced in \#3.  Also note that \#3 eventually has to apply at some point in the procedure, as $c_j>1$ for all $j$, and thus that an inversion is always eventually added.  As there is a unique allowable move at each step, the resulting 1-inversion tableau $\widetilde{T}$ is unique and the algorithm represents a well-defined map $\phi_2: S_0(\widetilde{\lambda}) \hookrightarrow S_1(\lambda)$.

$\phi_2$ is constructed so that it is clearly the inverse of $\phi_1$.  In particular, it is straightforward to check that $\phi_2 \circ \phi_1 (\tau) = \tau$ and $\phi_1 \circ \phi_2 (T) = T$  This demonstrates the bijectivity of both maps and allows us to conclude that $\vert S_1(\lambda) \vert = \vert S_0(\widetilde{\lambda}) \vert$.

\begin{figure}[h!]
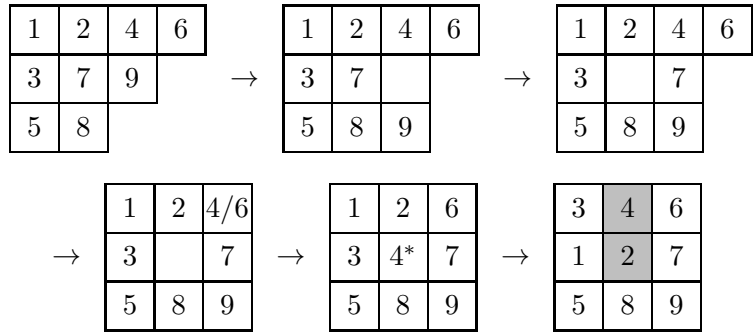

\centering
\begin{small}
\begin{ytableau}
1 & 2 & 4 & 6 \\
3 & 7 & 9 \\
5 & 8 
\end{ytableau}
\hspace{.03in}
\raisebox{-12pt}{$\rightarrow$}
\hspace{.03in}
\begin{ytableau}
1 & 2 & 4 & 6 \\
3 & 7 & \\
5 & 8 & 9
\end{ytableau}
\hspace{.03in}
\raisebox{-12pt}{$\rightarrow$}
\hspace{.03in}
\begin{ytableau}
1 & 2 & 4 & 6 \\
3 & & 7 \\
5 & 8 & 9
\end{ytableau}

\vspace{.15in}

\raisebox{-12pt}{$\rightarrow$}
\hspace{.03in}
\begin{ytableau}
1 & 2 & 4/6 \\
3 & & 7 \\
5 & 8 & 9
\end{ytableau}
\hspace{.03in}
\raisebox{-12pt}{$\rightarrow$}
\hspace{.03in}
\begin{ytableau}
1 & 2 & 6 \\
3 & 4^* & 7 \\
5 & 8 & 9
\end{ytableau}
\hspace{.03in}
\raisebox{-12pt}{$\rightarrow$}
\hspace{.03in}
\begin{ytableau}
3 & *(gray!50)4 & 6 \\
1 & *(gray!50)2 & 7 \\
5 & 8 & 9
\end{ytableau}
\end{small}
\caption{$S_0(4,3,2) \hookrightarrow S_1(3,3,3)$ Example}
\label{fig: rectangular i=0 to i=1 inverse example}
\end{figure}
\end{proof}

\subsection{Calculating $\vert S_1(\lambda) \vert$, Non-Rectangular Case}
\label{subsec: enumeration, general i=1}

The goal of this subsection is to prove a generalization of Theorem \ref{thm: enumeration, rectangular i=1} that enumerates 1-inverted Young tableaux for an arbitrary (non-rectangular) shape $\lambda$.  Before explicitly giving our result in Theorem \ref{thm: enumeration, general i=1}, we need to establish some new notation describing non-rectangular tableax shapes.

For a tableau of shape $\lambda = (\lambda_1,\hdots,\lambda_m)$, define $d_i = \lambda_i-\lambda_{i+1}$ for $1 \leq i < m$ and $d_m = \lambda_m$. Notice that $d_i \geq 0$ for all $i$ and that $d_i > 0$ iff the rightmost entry in the $i$\textsuperscript{th} row of our tableau is a ``lower-right corner" for our tableau.  Similarly define $\widetilde{d}_i = \lambda_{i-1}-\lambda_i$ for $1 < i \leq m$ and $\widetilde{d}_1 = \infty$.  Notice that $\widetilde{d}_i \geq 0$ for all $i$ and that $\widetilde{d}_i > 0$ iff an additional box may be added to the end of the $i$\textsuperscript{th} row of a tableau of shape $\lambda$ to produce a valid tableau shape. 

\begin{theorem}
\label{thm: enumeration, general i=1}
Let $\lambda=(\lambda_1,\hdots,\lambda_m)$ be a non-decreasing sequence of positive integers with $m \geq 1$.  Then:
\begin{center}
$\displaystyle{\vert S_1(\lambda_1,\hdots,\lambda_m) \vert = \sum_E \vert S_0(\lambda_1 + \epsilon_1,\hdots, \lambda_m + \epsilon_m) \vert}$
\end{center}

\noindent Where the summation runs over all tuples $E = (\epsilon_1,\hdots,\epsilon_m)$ such that $\epsilon_i= 0,\pm 1$ for all $i$, $\epsilon_i = -1$ for precisely one $i = i_1$ such that $i_1 >1$ and $d_{i_1}>0$, and $\epsilon_i = 1$ for precisely one $i = i_2$ such that $i_2 < i_1$ and $\widetilde{d}_{i_2}>0$.
\end{theorem}

Observe that valid tuples $E=(\epsilon_1,\hdots,\epsilon_m)$ include all zeroes apart from a single $+1$ and a single $-1$, with the $-1$ appearing after the $+1$.  In terms of actual tableaux, the summation of Theorem \ref{thm: enumeration, general i=1} runs over all shapes that are obtained from $\lambda$ by moving a single lower-right corner (that isn't in the first row) to create a new lower-right corner that is at least one row higher than the box's original location.  See Figure \ref{fig: general shape change comparison} for an example of the relevant tableaux shapes in the summation.

\begin{figure}[h!]
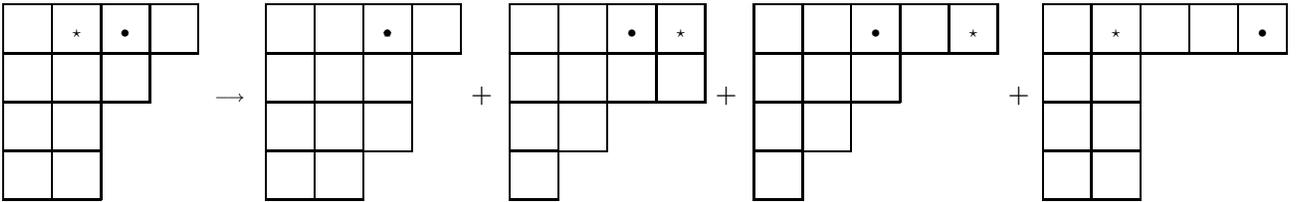

\centering
\begin{tiny}
\ydiagram{4,3,2,2}
*[\bullet]{0,2+1}
*[\star]{0,0,0,1+1}
\hspace{.05in}
\raisebox{-18.5pt}{$\longrightarrow$}
\hspace{.04in}
\ydiagram{4,3,3,2}
*[\bullet]{0,2+1}
*[\star]{0,0,2+1}
\hspace{.02in}
\raisebox{-18.5pt}{\scalebox{1.5}{$+$}}
\hspace{.02in}
\ydiagram{4,4,2,1}
*[\bullet]{0,2+1}
*[\star]{0,3+1}
\hspace{.02in}
\raisebox{-18.5pt}{\scalebox{1.5}{$+$}}
\hspace{.02in}
\ydiagram{5,3,2,1}
*[\bullet]{0,2+1}
*[\star]{4+1}
\hspace{.02in}
\raisebox{-18.5pt}{\scalebox{1.5}{$+$}}
\hspace{.0in}
\ydiagram{5,2,2,2}
*[\bullet]{4+1}
*[\star]{0,0,0,1+1}
\end{tiny}
\caption{Shape Change in the Bijection of Theorem \ref{thm: enumeration, general i=1} when $\lambda = (4,3,2,2)$}
\label{fig: general shape change comparison}
\end{figure}

Also notice that when $\lambda=(n,\hdots,n)$ there is a single valid tuple $E = (1,0,\hdots,0,-1)$, and thus that the summation of Theorem \ref{thm: enumeration, general i=1} specializes to the equality of Theorem \ref{thm: enumeration, rectangular i=1} in the case of rectangular tableaux.

\begin{proof}[Proof of Theorem \ref{thm: enumeration, general i=1}]
\noindent $\mathbf{\left( \hookrightarrow \right)}$ Take $\tau \in S_1(\lambda_1,\hdots,\lambda_m)$, and assume that the sole inversion pair $(a,b)$ lies in the $k$\textsuperscript{th} column of $\tau$.  We repeatedly ``bump" a sequence of distinguished elements rightward, as in the first map $\phi_1$ from the proof of Theorem \ref{thm: enumeration, rectangular i=1}, following the exact same rules at each intermediate step.  The difference here is that we terminate the procedure (possibly before the rightmost column) as soon as a distinguished element can be placed in a valid tableau position that does not fit inside the original shape $(\lambda_1,\hdots,\lambda_m)$.  The resulting tableau shape will differ from $\lambda$ in that it will have a new lower-right corner added at the final step, and that it will have a single vacated box at the bottom of some column ($k$\textsuperscript{th} or leftward) that is created by the ``back-filling" component of the procedure.  The specific location of these boxes depends on the inverted tableau, but the procedure ensures that it will always be one of the shapes from the summation in the Theorem.  As in the proof of Theorem \ref{thm: enumeration, rectangular i=1} this is a well-defined map that always yields a standard Young tableau.

\noindent $\mathbf{\left( \hookleftarrow \right)}$ Now take $\tau \in S_0(\widetilde{\lambda})$, where $\widetilde{\lambda}$ is any one of the shapes from the summation in the Theorem, and consider the entry in the sole box (always one of the lower-right corners of $\widetilde{\lambda}$) that does not fit inside the shape $\lambda$.  As in the second map $\phi_2$ from the proof of Theorem \ref{thm: enumeration, rectangular i=1}, we recursively ``back-bump" a sequence of distinguished elements one column leftward, following the exact same rules as in that earlier procedure.  Now the difference is that we are only allowed to slide entries into boxes that are contained in $\lambda$, even if that spot is a valid tableau position.  For the same reasons as described in the proof of Theorem \ref{thm: enumeration, rectangular i=1}, this map is well-defined and produces a 1-inversion tableau of the desired shape.  This map is also clearly the inverse of the map defined above, so if you define a piecewise map by ranging over all valid shapes $\widetilde{\lambda}$ the result is a bijection.
\end{proof}

\section{Enumerating $i$-Inverted Young Tableaux, $i > 1$}
\label{sec: enumeration i>1}

In Section \ref{sec: enumeration i=1} we developed bijections between 1-inverted rectangular tableaux of an arbitrary shape $\lambda$ and 0-inverted tableaux of some related collection of shapes $\lbrace \widetilde{\lambda}_i \rbrace$.  One may ask whether the methods of that section readily extend to higher numbers?  More specifically, does there exist a clear relationship between the number of $i$-inverted tableaux of shape $\lambda$ and the numbers of $(i-j)$-inverted tableaux of shapes $\lbrace \widetilde{\lambda}_i \rbrace$ for some integer $j>0$?

In Appendix \ref{sec: inversion tables}, we present a full comparison of $\vert S_i \vert$ for several choices of three-row rectangular shapes $\lambda=(n,n,n)$ and the associated ``stair-step" shapes $\widetilde{\lambda}=(n+1,n,n-1)$.  As predicted by Theorem \ref{thm: enumeration, rectangular i=1}, we have $\vert S_1(\lambda) \vert = \vert S_0(\widetilde{\lambda}) \vert$.  Sadly, even in our specialization to rectangular $\lambda$, Theorem \ref{thm: enumeration, rectangular i=1} does not extend to a bijection between $\vert S_i(\lambda) \vert$ and $\vert S_{i-1}(\widetilde{\lambda}) \vert$ for all $i > 1$.

What those tables do reveal are discernible patterns in $\vert S_i(\lambda) \vert$ for sufficiently high choices of $i$.  In addition to $\vert S_{M_\lambda}(\lambda) \vert = 1$ for the maximum inversion number $M_\lambda$ (see Proposition \ref{thm: maximal inversion number}), notice that $\vert S_{M_\lambda - 1}(\lambda) \vert = 3n -1$ and $\vert S_{M_\lambda - 2}(\lambda) \vert = \frac{(3n-2)(3n+1)}{2}$ for all $n \geq 1$.  Also notice that $\vert S_i(\lambda) \vert = \vert S_{i-2}(\widetilde{\lambda}) \vert$ for all values of $i$ beyond a certain point.  We henceforth refer to the region where this latter equality holds as the ``tail-end" of the inversion table.  For our three-row rectangular tableaux with $\lambda = (n,n,n)$, this tail-end consistently begins at $i = n+1$.

In this section we prove formulas for $\vert S_{M_\lambda - 1}(\lambda) \vert$ and $\vert S_{M_\lambda - 2}(\lambda) \vert$ in the rectangular case, generalizing our observations about the three-row tableaux in Appendix \ref{sec: inversion tables}.  We then prove a special case of the ``tail-end" bijection and conjecture as to how this result extends to the general rectangular case.  For this entire section, we restrict our attention to rectangular tableaux, although the authors suspect that there may exist generalizations to non-rectangular tableaux that mirror Subsection \ref{subsec: enumeration, general i=1}.

\subsection{Calculating $\vert S_{M_\lambda - 1}(\lambda) \vert$ and $\vert S_{M_\lambda - 2}(\lambda) \vert$, Rectangular Case}
\label{subsec: M-1,M-2}

Directly enumerating inverted tableaux becomes increasingly tractable as one approaches the maximum inversion number $M_\lambda$.  This is because having large numbers of inversions in a fixed column increases the likelihood that the tableau splits after that column.  In the $(M_\lambda - 1)$- and $(M_\lambda - 2)$-inversion cases, non-split tableaux are so rare that they may be easily enumerated.  As split tableaux are easily counted, this allows for relatively straightforward closed formulas for $\vert S_{M_\lambda - 1}(\lambda) \vert$ and $\vert S_{M_\lambda - 2}(\lambda) \vert$.

\begin{theorem}
\label{thm: M-1 rectangular}
Take $n \geq 1$ and $m \geq 2$, and consider the $m$-row shape $\lambda = (n,\hdots, n)$.  If $M_\lambda = n T_{m-1}$ is the maximum inversion number, then $\vert S_{M_\lambda - 1}(\lambda) \vert = mn-1$.
\end{theorem}
\begin{proof}
Fix any $m \geq 2$.  With our choice of $m$ implicit, temporarily denote $S_i(n) = S_i(n,\hdots,n)$ and $M_n = M_{(n,\hdots,n)} = nT_{m-1}$.  Thus we are looking to show that $\vert S_{M_n -1}(n) \vert = mn-1$.  We proceed by induction on $n \geq 1$, with the base case of $\vert S_{M_1-1}(1)  \vert = m-1$ following from Subsection \ref{subsec: enumerating one-column inverted tableaux} and the Mahonian number identity $M(m-1,T_{m-1}-1) = m-1$.

For the inductive case, we think about isolating the leftmost column in an $(M_n -1)$-inverted tableau of size $m \times n$.  To obtain the required number of inversions, one of the following must hold:
\begin{enumerate}
\item The leftmost column has precisely $M_1$ inversions and the remaining $(n-1)$ columns have a total of precisely $M_{(n-1)} - 1$ inversions.
\item The leftmost column has precisely $M_1 -1$ inversions and the remaining $(n-1)$ columns have a total of precisely $M_{(n-1)}$ inversions.
\end{enumerate}

In case \#1, there is a unique ordering of the elements in the leftmost column (relative to the column immediately on its right) and the tableau necessarily splits after the first column.  It follows that the number of tableaux satisfying this case equals $\vert S_{M_{(n-1)}-1}(n-1) \vert$.

In case \#2, there are two sub-cases depending upon whether the largest entry $a_m$ or the second largest entry $a_{m-1}$ in the leftmost column is directly to the left of the smallest entry $b_1$ in the column immediately on its right.  For an illustration of these two sub-cases when $m=4$, see Figure \ref{fig: M-1 example} below.

In the first sub-case, the tableau must split after the first column, so that the first column must contain the entries $1,2,\hdots,m$.  Since there is a unique way to produce $M_{(n-1)}$ inversions across the latter columns, the placement of $m+1,m+2,\hdots,mn$ is predetermined.  In the first column, there are $m-1$ inversions involving $a_m$ and we must have $T_{m-1} - 1 - (m-1) = T_{m-2} - 1$ inversions among the remaining $m-1$ entries in that column.  There are $M(m-2,T_{m-2}-1) = m-2$ distinct arrangements that produce these $T_{m-2}-1$ inversions.  This means that there are precisely $m-2$ inverted tableaux satisfying this sub-case.

In the second sub-case, the tableau need not split after the first column, as the largest entry $a_m$ in the leftmost column may or may not be greater than the smallest entry $b_1$ in the column directly to its right.  No matter the relationship between these two entries, there is a unique arrangement that produces the requisite $M_{(n-1)}$ inversions across the latter columns.  In the first column, there are always $m-2$ inversions involving the top entry $a_{m-1}$, leaving $T_{m-2}$ inversions among the remaining $m-1$ entries in the column.  There is only $M(m-2,T_{m-2})=1$ arrangement of the remaining entries in the leftmost column that produces these $T_{m-2}$ inversions.  Also notice that the required arrangement in the leftmost column ensures that $a_m$ is directly to the left of the second-smallest entry $b_2$ of the next column, and that the tableau always splits after the second column.  This leaves the relationship between $a_m$ and $b_1$ as the only open question about such a tableau.  It follows that there are precisely $2$ inverted tableaux in this sub-case (one for $a_m < b_1$, one for $a_m > b_1$).  

Collecting results from all sub-cases and then applying the inductive assumption gives the desired result: $\vert S_{M_n}(n) \vert = \vert S_{M_{n-1}}(n-1) \vert + (m-2) + 2= (m(n-1) - 1) + (m - 2) + 2 = mn - 1$.
\end{proof}

\begin{figure}[h!]
\centering
\begin{ytableau}
a_4 & b_1 \\
a_3 & b_2 \\
a_2 & b_3 \\
a_1 & b_4 
\end{ytableau}
\hspace{.7in}
\begin{ytableau}
a_4 & b_1 \\
*(gray!50) & b_2 \\
*(gray!50)(2) & b_3 \\
*(gray!50) & b_4 
\end{ytableau}
\hspace{.2in}
\begin{ytableau}
a_3 & b_1 \\
a_4 & b_2 \\
a_2 & b_3 \\
a_1 & b_4
\end{ytableau}

\caption{The leftmost columns of an $(M_\lambda -1)$-inverted 4-row tableau, with subscripts denoting the relative ordering within each column (we've reordered the second column so it appears increasing).  Case \#1 from the proof of Theorem \ref{thm: M-1 rectangular} is on the left, whereas the two sub-cases of case \#2 from the proof are on the right.  Parentheses denote number of inversions required in shaded sub-tableau.}
\label{fig: M-1 example}
\end{figure}

\begin{theorem}
\label{thm: M-2 rectangular}
Take $n \geq 1$ and $m \geq 3$, and consider the $m$-row shape $\lambda = (n,\hdots, n)$.  If $M_\lambda = n T_{m-1}$ is the maximum inversion number, then $\vert S_{M_\lambda - 2}(\lambda) \vert = \displaystyle{\frac{(mn-2)(mn+1)}{2}}$.
\end{theorem}
\begin{proof}
This requires a more complicated variation on the method from the proof of Theorem \ref{thm: M-2 rectangular}.  Once again we fix $m \geq 3$ and temporarily denote $S_i(n) = S_i(n,\hdots,n)$, $M_n = M_{(n,\hdots,n)} = nT_{m-1}$.  In this case we are looking to show that $\vert S_{M_n-2}(n) \vert = \frac{(mn-2)(mn+1)}{2}$.  We proceed by induction on $n \geq 1$.

For the base case $\vert S_{M_n-2}(1) \vert$, one may directly verify the Mahonian number $M(m-1,T_{m-1}-2) = \frac{(m-2)(m+1)}{2}$. For the inductive case we once again isolate the leftmost column in an $(M_n - 2)$-inverted tableau of size $m \times n$.  We now have three ways to obtain the required number of inversions:

\begin{enumerate}
\item The leftmost column has precisely $T_m = M_1$ inversions and the remaining $(n-1)$ columns have precisely $M_{(n-1)}-2$ inversions.
\item The leftmost column has precisely $T_m - 1 = M_1 - 1$ inversions and the remaining $(n-1)$ columns have precisely $M_{(n-1)}-1$ inversions.
\item The leftmost column has precisely $T_m = M_1 - 2$ inversions and the remaining $(n-1)$ columns have precisely $M_{(n-1)}$ inversions.
\end{enumerate}

In case \#1, there is a unique ordering of the leftmost column to obtain the maximal number of $M_1$ inversions and the tableau necessarily splits after the first column.  It follows that the number of inverted tableaux in this case equals $\vert S_{M_{(n-1)}-2}(n-1) \vert$.

In case \#2, there are two sub-cases depending upon whether the largest entry $a_m$ or the second largest entry $a_{m-1}$ in the leftmost column is directly left of the smallest entry $b_1$ in the second column.  For an illustration of these sub-cases when $m = 4$, see Figure \ref{fig: M-2 example}.

In the first sub-case of case \#2, the tableau must split after the first column.  In the first column, $a_m$ participates in $m-1$ inversions and thus the remaining elements must account for $T_{m-1} - 1 - (m - 1) = T_{m-2} - 1$ total inversions.  By Theorem \ref{thm: M-1 rectangular}, there are $m - 2$ ways to achieve these $T_{m-2} - 1$ inversions.  It follows that there are precisely $(m-2) \vert S_{M_{(n-1)}-1}(n-1) \vert$ tableaux in this sub-case.

In the second sub-case of case \#2, the tableau needn't split after the first column.  In the first column, $a_{m-1}$ accounts for $m-2$ inversions and hence the remaining elements in that column must account for $T_{m-2}$ inversions.  This implies that the rest of the first column must be ``fully inverted", and thus that $a_m$ is directly to the left of the second-smallest entry $b_2$ in the second column.  We then have $a_{m-1} < b_1$ and $a_m < b_2$, leaving the relationship between $a_m$ and $b_1$ as the only question about the first two columns.  Also noticing that $b_1$ is directly to the left of either $c_m$ or $c_{m-1}$ in the third column, as $m \geq 3$ we know that $a_m < b_2 < c_2 \leq c_{m-1}/c_m$.  This means we needn't worry about the relationship between $a_m$ and any elements of the third column.  For each possible arrangement of the tableau's final $n-1$ columns, it follows that there are precisely two inverted tableaux fitting this sub-case (one for $a_m < b_1$, one for $a_m > b_1$).  Thus there are precisely $2 \vert S_{M_{(n-1)}-1}(n-1) \vert$ tableaux in this sub-case.

In case \#3 there are four sub-cases, each depending upon which elements in the first column are directly to the left of the smallest two entries $b_1,b_2$ in the second column.  Once again, reference Figure \ref{fig: M-2 example} for an illustration of these sub-cases in the situation where $m=4$.  Notice that in all of these sub-cases, the $M_{(n-1)}$ inversions needed over the final $n-1$ columns ensures that there is precisely one arrangement of those columns for each arrangement of the first column.  It also guarantees that any such tableau must split after the second column.

In the first sub-case of case \#3, $a_m$ is directly to the left of $b_1$ and the tableau must split after the first column.  In the first column, $a_m$ participates in $m-1$ inversions and thus the remaining elements account for $T_{m-1}-2-(m-1) = T_{m-2}-2$ total inversions.  By the base case of this theorem, there are $\frac{(m-1-2)(m-1+1)}{2} = \frac{(m-3)m}{2}$ ways to achieve these inversions.  Thus there are exactly $\frac{(m-3)m}{2}$ tableaux in this sub-case.

In the second sub-case of case \#3, $a_{m-1}$ is directly to the left of $b_1$ and $a_m$ is directly to the left of $b_2$.  As $a_{m-1}$ is involved in $m-2$ inversions and $a_m$ is involved in $m-2$ inversions, this leaves $T_{m-1} - 2 - (m-2) - (m-2) = T_{m-3} - 1$ inversions required across the remaining $m-2$ elements of that column.  By Theorem \ref{thm: M-1 rectangular}, there are precisely $(m-2) - 1$ distinct ways to accomplish these inversions.  For each of these $m-3$ arrangements, there is a single open choice as to the relationship between $a_m$ and $b_1$, seeing as $a_m$ is to the left of $b_2$ and $a_{m-1}$ is to the left of $b_1$.  Thus there are a total of $2 (m-3)$ tableaux in this sub-case.

In the third sub-case of case \#3, $a_{m-1}$ still lies directly to the left of $b_1$ but now $a_{m-2}$ lies directly to the left of $b_2$.  In this situation, $a_{m-1}$ is involved in $m-2$ inversions and $a_m$ is involved in $m-3$ inversions, which leaves $T_{m-1} - 2 - (m-2) - (m-3) = T_{m-3}$ inversions to be accounted for over the remaining $m-2$ entries of the first column.  There is precisely one way to account for these remaining inversions, with $a_m$ necessarily appearing directly to the left of $b_3$.  With $a_{m-1}$ to the left of $b_1$ and $a_m$ to the left of $b_3$, the relationship of $a_m$ to both $b_1$ and $b_2$ is an open question.  Here the possible orderings are $a_m < b_1 < b_2$, $b_1 < a_m < b_2$, and $b_1 < b_2 < a_m$.  As this is the only undetermined aspect of these tableaux, there are precisely $3$ tableaux in this sub-case.

In the final sub-case of case \#3, $a_{m-2}$ lies directly to the left of $b_1$.  In the first column, $a_{m-2}$ is involved in $m-3$ inversions, leaving $T_{m-1} - 2 - (m-3) = T_{m-2}$ inversions for the remaining $m-1$ entries of that column.  There is precisely one arrangement of those entries that gives $T_{m-2}$ inversions, but in this situation we still need to determine the relationship of both $a_m$ and $a_{m-1}$ to $b_1$.  Here there are possible orderings $a_{m-1} < a_m < b_1$, $a_{m-1} < b_1 < a_m$, and $b_1 < a_{m-1} < a_m$.  As this is the only undetermined aspect of these tableaux, there are precisely $3$ tableaux in this sub-case.

Collecting all sub-cases and citing Theorem \ref{thm: M-1 rectangular} along with the inductive assumption gives:

\begin{center}
$\vert S_{M_{(n-1)}-2}(n-1) \vert + (m-2) \vert S_{M_{(n-1)}-1}(n-1) \vert + 2 \vert S_{M_{(n-1)}-1}(n-1)\vert + \frac{1}{2}(m-3)m + 2(m-3) + 6$

\vspace{.1in}

$\displaystyle{= \hspace{.1in} \frac{(m(n-1)-2)(m(n-1)+1)}{2} + m(m(n-1) - 1) + \frac{m(m+1)}{2} \hspace{.1in}=}$ 

\vspace{.1in}

$\displaystyle{\frac{m^2 n^2 -m^2 n + mn - m^2 n + m^2 - m - 2mn + 2m -2}{2} + \frac{2m^2 n - 2m^2 - 2m}{2} + \frac{m^2 + m}{2}}$ 

\vspace{.1in}

$\displaystyle{= \hspace{.1in}\frac{m^2 n^2 - mn - 2}{2} \hspace{.1in}= \hspace{.1in}\frac{(mn-2)(mn+1)}{2}}$
\end{center}
\end{proof}

\begin{figure}[h!]
\centering
\begin{ytableau}
a_4 & b_1 \\
a_3 & b_2 \\
a_2 & b_3 \\
a_1 & b_4 
\end{ytableau}
\hspace{.4in}
\begin{ytableau}
a_4 & b_1 \\
*(gray!50) & b_2 \\
*(gray!50)(2) & b_3 \\
*(gray!50) & b_4 
\end{ytableau}
\hspace{.2in}
\begin{ytableau}
a_3 & b_1 \\
a_4 & b_2 \\
a_2 & b_3 \\
a_1 & b_4
\end{ytableau}
\hspace{.4in}
\begin{ytableau}
a_4 & b_1 \\
*(gray!50) & b_2 \\
*(gray!50) (1)& b_3 \\
*(gray!50) & b_4
\end{ytableau}
\hspace{.2in}
\begin{ytableau}
a_3 & b_1 \\
a_4 & b_2 \\
a_1 & b_3 \\
a_2 & b_4
\end{ytableau}
\hspace{.2in}
\begin{ytableau}
a_3 & b_1 \\
a_2 & b_2 \\
a_4 & b_3 \\
a_1 & b_4
\end{ytableau}
\hspace{.2in}
\begin{ytableau}
a_2 & b_1 \\
a_4 & b_2 \\
a_3 & b_3 \\
a_1 & b_4
\end{ytableau}
\caption{The leftmost columns of an $(M_\lambda - 2)$-inverted 4-row tableau, with subscripts denoting relative ordering within each column.  Case \#1 from the proof of Theorem \ref{thm: M-2 rectangular} is on the left, the two sub-cases of case \#2 are in the middle, and the four sub-cases of case \#3 are on the right.  Parentheses denote number of inversions required in shaded sub-tableaux.}
\label{fig: M-2 example}
\end{figure}

Notice that the $m \geq 3$ condition of Theorem \ref{thm: M-2 rectangular} (as opposed to the $m \geq 2$ condition of Theorem \ref{thm: M-1 rectangular}) sidesteps the semantic difficulties arising from the fact that a column in a two-row tableau can have at most $1$ inversions, which implies that $M_1 - 2$ would need to be negative.  The fact that $m \geq 3$ was also required in the appeal to which entry $c_i$ lay directly to the right of $b_1$ in case \#2.  This precluded the situation where $a_m$ may be smaller than elements of the third columns, which would have necessitated an even more sophisticated consideration of sub-cases.

In theory, the technique of Theorems \ref{thm: M-1 rectangular} and \ref{thm: M-2 rectangular} could be extended to the $M_\lambda - i$ inversion case is we require rectangular tableaux with at least $m = i + 1$ columns.  However, the massive number of sub-cases required would make the proof of even the $M_\lambda - 3$ case extremely daunting.

\subsection{Calculating $\vert S_i(\lambda) \vert$, ``Tail-End"}
\label{subsec: enumeration tail-end}

Consider the three-row rectangular shape $\lambda = (n,n,n)$ and let $\widetilde{\lambda} = (n+1,n,n-1)$ be the associated stair-step shape, as in the proofs of Section \ref{sec: enumeration i=1}.  Appendix \ref{sec: inversion tables} reveals that $\vert S_i(\lambda) \vert =\vert S_{i-2}(\widetilde{\lambda}) \vert$ when $i > n$.  Since Proposition \ref{thm: maximal inversion number} gives $M_{\widetilde{\lambda}} = M_\lambda - 2$, it follows that the inversion tables for $\lambda$ and $\widetilde{\lambda}$ can be matched ``from the bottom up" until we reach $i=n$.  Once again, we refer to this range as the ``tail-end" of the inversion table for $\lambda$.  In this Subsection, we take steps towards generalized proof of this ``tail-end" phenomenon.

So consider the $m$-row rectangular shape $\lambda = (n,\hdots,n)$ and the associated stair-step shape $\widetilde{\lambda} = (n+1,n,\hdots,n,n-1)$.  By Proposition \ref{thm: maximal inversion number} we have $M_{\widetilde{\lambda}} = M_\lambda - (m-1)$.  Thus one would expect an equality $\vert S_i(\lambda) \vert = \vert S_{i-m+1}(\widetilde{\lambda}) \vert$ for ``sufficiently large" $i$.  Predicting what constitutes ``sufficiently large" $i$ in this generalized tail-end is far less obvious.  To find appropriate bounds on $i$, as well as to develop a general method of proof, we begin with a consideration of one-column tableaux:

\begin{lemma}
\label{thm: tail-end lemma}
Let $m \geq 2$, and consider the m-row column tableaux shape $\lambda = (1,\hdots,1)$.  For the (m-1)-row ``hook" shape $\widetilde{\lambda} = (2,1,\hdots,1)$, we have $\vert S_i(\lambda) \vert = \vert S_{i-m+1}(\widetilde{\lambda})\vert$ for all $i > T_{m-2}$, where $T_{m-2}$ is the triangular number.  Moreover, there exists a bijection $\phi: S_i(\lambda) \rightarrow S_{i-m+1}(\widetilde{\lambda})$ such that, if the top entry of $\tau \in S_i(\lambda)$ is $k$, then the sole entry in the second column of $\phi(\tau)$ is also $k$.
\end{lemma}

\begin{proof}
For each $1 \leq k \leq m$, denote the set of inverted tableaux in $S_i(\lambda)$ whose first entry is $k$ by $S_i^k(\lambda)$.  Similarly, denote the set of inverted tableaux in $S_i(\widetilde{\lambda})$ whose sole second column entry is $k$ by $S_i^k(\widetilde{\lambda})$.  To expedite notation, for the $p$-row column tableaux we refer to our $p \times 1$ shape as $(1,\hdots,1) = 1^p$, so that our original shape is $\lambda = 1^m$.

So fix any $1 \leq k \leq m$ and enforce our condition that $i > T_{m-2}$.  This is equivalent to saying that $i-m+2 > T_{m-3}$.  As the maximum number of inversions for a column tableau of size $(m-2) \times 1$ is $M_{m-2} = T_{m-3}$, it follows that $\vert S_{i-k+1}(1^{m-2}) \vert = \hdots = \vert S_{i-m+2}(1^{m-2}) \vert = 0$ no matter our choice of $k$.  This ensures:
\begin{equation}
\label{eq: tail-end lemma A}
\vert S_{i-k+1}(1^{m-2}) \vert + \vert S_{i-k}(1^{m-2}) \vert + \hdots + \vert S_{i-k-m+3}(1^{m-2}) \vert = \vert S_{i-m+1}(1^{m-2}) \vert + \hdots + \vert S_{i-k-m+3}(1^{m-2}) \vert
\end{equation}

Consider any inverted tableau $\tau \in S_{i-k+1}(1^{m-1})$, and assume that the top entry of $\tau$ is $a$.  $\tau$ contains precisely $a-1$ inversion pairs that involve $a$.  Removing $a$ and re-indexing the remaining entries of $\tau$ so that $a$ isn't skipped then produces an element of $S_{i-k+1-(a-1)}(1^{m-2})$.  Ranging over all possible top entries for elements of $S_{i-k+1}(1^{m-1})$ gives:
\begin{equation}
\label{eq: tail-end lemma B}
\vert S_{i-k+1}(1^{m-1}) \vert = \vert S_{i-k+1}(1^{m-2}) \vert + \vert S_{i-k}(1^{m-2}) \vert + \hdots + \vert S_{i-k+1-(m-2)}(1^{m-2}) \vert 
\end{equation}

Applying similar reasoning to $S_{i-m+1}(1^{m-1})$ then allows us to rewrite Equation \ref{eq: tail-end lemma A} as:
\begin{equation}
\label{eq: tail-end lemma C}
\vert S_{i-k+1}(1^{m-1}) \vert = \vert S_{i-m+1}(1^{m-1}) \vert - \vert S_{i-m-k+2}(1^{m-2}) \vert - \vert S_{i-m-k+1}(1^{m-2}) \vert - \hdots \vert S_{i-2m+3}(1^{m-2}) \vert
\end{equation}

Pause to consider $S_i^k(\lambda) = S_i^k(1^{m})$.  Similarly to above, any $\tau \in S_i^k(1^{m})$ has $k-1$ inversion pairs involving $k$, so that if we simply remove $k$ (and re-index so that we don't skip $k$) we have an element of $S_{i-k+1}(1^{m-1})$.  It follows that:
\begin{equation}
\label{eq: tail-end lemma D}
\vert S_i^k(1^{m}) \vert = \vert S_{i-k+1}(1^{m-1}) \vert
\end{equation}

Now consider $S_{i-m+1}^k(\widetilde{\lambda})$, and compare this set to $S_{i-m+1}(1^{m-1})$ by removing the sole entry in the second column (and then re-indexing so $k$ isn't skipped).  Note that removing this entry does not effect the number of inversions in the tableau.  In general, $\vert S_{i-m+1}(1^{m-1}) \vert \geq \vert S_{i-m+1}^k(\widetilde{\lambda}) \vert$, as simply adding $k$ to the right of the first entry in a element from $S_{i-m+1}(m-1)$ may result in a tableau that isn't row-standard.  Eliminating elements of $S_{i-m+1}(1^{m-1})$ that fail to be row-standard after inserting $k$ gives $\vert S_{i-m+1}^k(\widetilde{\lambda}) \vert = \vert S_{i-m+1}(1^{m-1}) \vert - \vert S_{i-m+1}^{k+1}(1^{m-1}) \vert - \hdots - \vert S_{i-m+1}^m(1^{m-1}) \vert$.  This gives:
\begin{equation}
\label{eq: tail-end lemma E}
\vert S_{i-m+1}^k(\widetilde{\lambda}) \vert = \vert S_{i-m+1}(1^{m-1}) \vert - \vert S_{i-m-k+2}(1^{m-2}) \vert - \vert S_{i-m-k+1}(1^{m-2}) \vert - \hdots - \vert S_{i-2m+3}(1^{m-2}) \vert
\end{equation}

Inserting Equations \ref{eq: tail-end lemma D} and \ref{eq: tail-end lemma E} into Equation \ref{eq: tail-end lemma C} yields:
\begin{equation}
\label{eq: tail-end lemma F}
\vert S_i^k(1^{m}) \vert = \vert S_{i-m+1}^k(\widetilde{\lambda}) \vert
\end{equation}

As Equation \ref{eq: tail-end lemma F} holds for all $1 \leq k \leq m$, it is possible to construct a bijection $\phi: S_i(\lambda) \rightarrow S_{i-m+1}(\widetilde{\lambda})$ that takes the top entry of each element of $S_i(\lambda)$ to the sole entry in the second column of an element from $S_{i-m+1}(\widetilde{\lambda})$
\end{proof}

Lemma \ref{thm: tail-end lemma} suggests that any bijection $\vert S_i(\lambda) \vert = \vert S_{i-m+1}(\widetilde{\lambda}) \vert$ requires the existence of a column in the $\lambda$-shaped tableau with greater than $T_{m-2}$ inversions.  Notice that this is in agreement with what Appendix \ref{sec: inversion tables} suggests about the tail-end in the $m=3$ row case: the condition that $i > n$ is precisely the number of inversions necessary to guarantee that an $i$-inverted tableau of shape $\lambda$ possesses at least one column with greater than $T_{m-2} = 1$ inversions.  For an $m$-row rectangular tableau, the number of inversions necessary to guarantee that every $i$-inverted tableau has a column with greater than $T_{m-2}$ inversions is $i > n T_{m-2}$.  These observations lead us to postulate the following:

\begin{conjecture}
\label{thm: enumeration, rectangular i big}
Let $n \geq 1$, $m \geq 2$, and take the m-row shapes $\lambda = (n,\hdots,n)$, $\widetilde{\lambda}(n+1,n,\hdots,n,n-1)$.  Then $\vert S_i(\lambda) \vert = \vert S_{i-m+1}(\widetilde{\lambda}) \vert$ for all $i > n T_{m-2}$, where $T_{m-2}$ is the triangular number.
\end{conjecture}

It is the authors' belief that it may be possible to prove Conjecture \ref{thm: enumeration, rectangular i big} via a ``recursive bumping" algorithm superficially similar to the method of proof in Theorem \ref{thm: enumeration, rectangular i=1}.  In this modified algorithm, the $\phi_1: S_i(\lambda) \hookrightarrow S_{i-m+1}(\widetilde{\lambda}) $ map would begin in the rightmost column that contains greater than $n T_{m-2}$ inversions.  The first rightward ``bump" would entail a reordering of the initial column in a manner that is uniquely determine by the bijection of Lemma \ref{thm: tail-end lemma}.  Subsequent rightward ``bumps" would then be required the conserve inversion number in a manner that can be uniquely undone by some inverse algorithm $\phi_2 : S_{i-m+1}(\widetilde{\lambda}) \hookrightarrow S_i(\lambda)$.

The difficulty in proving Conjecture \ref{thm: enumeration, rectangular i big} derives from these later ``bumps".  In particular, if a tableau possesses additional inversions to the right of the starting column, those inversions make the preservation of inversion number very complicated when attempting to bump a distinguished element past their column.  Resolving the difficulties in this proof, as well as developing tractable methods for dealing with $\vert S_i(\lambda) \vert$ when $i$ is greater than $1$ but too small to be in the ``tail-end", are the most significant remaining questions about enumerating tableaux inversions.

\bibliographystyle{amsplain}


\newpage

\appendix
\section{Inversion Tables}
\label{sec: inversion tables}

In this section we summarize some results from a computer program written by the authors to generate inverted Young tableaux.  This is done by first generating all standard Young tableaux of a given shape.  Each of the columns of each tableau have their elements permuted, and if the resulting Young tableau is row-standard, the number of inversions are counted and aggregated across every generated standard Young tableaux.  We highlight some entries to draw attention to a potential bijection, given in Conjecture \ref{thm: enumeration, rectangular i big}.  In its current form, the program can handle Young tableaux of at most three rows.  A version of the program is available \href{https://sites.google.com/a/valpo.edu/beagley/research/YoungTableauxInversions.java}{here}\footnote{\url{https://sites.google.com/a/valpo.edu/beagley/research/YoungTableauxInversions.java}}, and questions about the program should be directed to the first author.

\vspace{.2in}

\begin{small}
\begin{center}
\begin{minipage}[t]{0.49\columnwidth}
\begin{tabular}{| >{$}c<{$} | >{$}c<{$} | >{$}c<{$} | >{$}c<{$} |}
\hline
& (2,2,2) & (3,2,1) & \\ \hline
m=0 & 5 & & \\ \hline
m=1 & \cellcolor{gray!50}16 & \cellcolor{gray!50}16 & m=0 \\ \hline
m=2 & 25 & & \\ \hline
m=3 & \cellcolor{gray!50}24 & \cellcolor{gray!50}24 & m=1 \\ \hline
m=4 & \cellcolor{gray!50}14 & \cellcolor{gray!50}14 & m=2 \\ \hline
m=5 & \cellcolor{gray!50}5 & \cellcolor{gray!50}5 & m=3 \\ \hline
m=6 & \cellcolor{gray!50}1 & \cellcolor{gray!50}1 & m=4 \\ \hline
\text{TOTAL} & 90 & 60 & \text{TOTAL} \\ \hline
\end{tabular}
\end{minipage}
\begin{minipage}[t]{0.49\columnwidth}
\begin{tabular}{| >{$}c<{$} | >{$}c<{$} | >{$}c<{$} | >{$}c<{$} |}
\hline
& (3,3,3) & (4,3,2) & \\ \hline
m=0 & 42 & & \\ \hline
m=1 & \cellcolor{gray!50}168 & \cellcolor{gray!50}168 & m=0 \\ \hline
m=2 & 330 & 366 & m=1 \\ \hline
m=3 & 414 & & \\ \hline
m=4 & \cellcolor{gray!50}357 & \cellcolor{gray!50}357 & m=2 \\ \hline
m=5 & \cellcolor{gray!50}222 & \cellcolor{gray!50}222 & m=3 \\ \hline
m=6 & \cellcolor{gray!50}103 & \cellcolor{gray!50}103 & m=4 \\ \hline
m=7 & \cellcolor{gray!50}35 & \cellcolor{gray!50}35 & m=5 \\ \hline
m=8 & \cellcolor{gray!50}8 & \cellcolor{gray!50}8 & m=6 \\ \hline
m=9 & \cellcolor{gray!50}1 & \cellcolor{gray!50}1 & m=7 \\ \hline
\text{TOTAL} & 1680 & 1260 & \text{TOTAL} \\ \hline
\end{tabular}
\end{minipage}

\vspace{.2in}

\noindent \begin{minipage}[t]{0.49\columnwidth}
\begin{tabular}{| >{$}c<{$} | >{$}c<{$} | >{$}c<{$} | >{$}c<{$} |}
\hline
& (4,4,4) & (5,4,3) & \\ \hline
m=0 & 462 & & \\ \hline
m=1 & \cellcolor{gray!50}2112 & \cellcolor{gray!50}2112 & m=0 \\ \hline
m=2 & 4785 & 5643 & m=1 \\ \hline
m=3 & 7051 & 7161 & m=2 \\ \hline
m=4 & 7436 & & \\ \hline
m=5 & \cellcolor{gray!50}5951 & \cellcolor{gray!50}5951 & m=3 \\ \hline
m=6 & \cellcolor{gray!50}3773 & \cellcolor{gray!50}3773 & m=4 \\ \hline
m=7 & \cellcolor{gray!50}1937 & \cellcolor{gray!50}1937 & m=5 \\ \hline
m=8 & \cellcolor{gray!50}803 & \cellcolor{gray!50}803 & m=6 \\ \hline
m=9 & \cellcolor{gray!50}263 & \cellcolor{gray!50}263 & m=7 \\ \hline
m=10 & \cellcolor{gray!50}65 & \cellcolor{gray!50}65 & m=8 \\ \hline
m=11 & \cellcolor{gray!50}11 & \cellcolor{gray!50}11 & m=9 \\ \hline
m=12 & \cellcolor{gray!50}1 & \cellcolor{gray!50}1 & m=10 \\ \hline
\text{TOTAL} & 34650 & 27720 & \text{TOTAL} \\ \hline
\end{tabular}
\end{minipage}
\begin{minipage}[t]{0.49\columnwidth}
\begin{tabular}{| >{$}c<{$} | >{$}c<{$} | >{$}c<{$} | >{$}c<{$} |}
\hline
& (5,5,5) & (6,5,4) & \\ \hline
m=0 & 6006 & & \\ \hline
m=1 & \cellcolor{gray!50}30030 & \cellcolor{gray!50}30030 & m=0 \\ \hline
m=2 & 75075 & 91520 & m=1 \\ \hline
m=3 & 123552 & 137137 & m=2 \\ \hline
m=4 & 148512 & 137163 & m=3 \\ \hline
m=5 & 138801 & & \\ \hline
m=6 & \cellcolor{gray!50}105495 & \cellcolor{gray!50}105495 & m=4 \\ \hline
m=7 & \cellcolor{gray!50}67158 & \cellcolor{gray!50}67158 & m=5 \\ \hline
m=8 & \cellcolor{gray!50}36297 & \cellcolor{gray!50}36297 & m=6 \\ \hline
m=9 & \cellcolor{gray!50}16667 & \cellcolor{gray!50}16667 & m=7 \\ \hline
m=10 & \cellcolor{gray!50}6448 & \cellcolor{gray!50}6448 & m=8 \\ \hline
m=11 & \cellcolor{gray!50}2065 & \cellcolor{gray!50}2065 & m=9 \\ \hline
m=12 & \cellcolor{gray!50}531 & \cellcolor{gray!50}531 & m=10 \\ \hline
m=13 & \cellcolor{gray!50}104 & \cellcolor{gray!50}104 & m=11 \\ \hline
m=14 & \cellcolor{gray!50}14 & \cellcolor{gray!50}14 & m=12 \\ \hline
m=15 & \cellcolor{gray!50}1 & \cellcolor{gray!50}1 & m=13 \\ \hline
\text{TOTAL} & 756756 & 630630 & \text{TOTAL} \\ \hline
\end{tabular}
\end{minipage}
\end{center}
\end{small}

\end{document}